\documentclass[Accepted,12pt]{elsarticle}
\usepackage{amsfonts}
\usepackage{amsmath,amssymb,amscd,enumerate,latexsym,graphicx}

 \usepackage{graphics}
\usepackage{graphicx}
 \usepackage{epsfig}
\usepackage{amssymb}
\usepackage{amsthm}

\newtheorem{thm}{Theorem}

\newdefinition{rmk}{Remark}
\newproof{pf}{Proof}
\newdefinition{example}{Example}
\newdefinition{definition}{Definition}
\newdefinition{property}{Property}
\newdefinition{problem}{Problem}
\newdefinition{corollary}{Corollary}

\journal{Arkiv}

\begin{document}

\begin{frontmatter}



\title{$C$-eigenvalues intervals for Piezoelectric-type tensors}


\author[label2]{Chaoqian Li}
\ead{lichaoqian@ynu.edu.cn}
\author[label2]{Yaotang Li\corref{cor1}}
\ead{liyaotang@ynu.edu.cn} \cortext[cor1]{Corresponding author.}
\address[label2]{School of Mathematics and Statistics, Yunnan
University, Kunming, P. R. China 650091}

\begin{abstract}
$C$-eigenvalues of piezoelectric-type tensors which are real and always exist, are introduced by Chen et al. \cite{Ch}.
And the largest $C$-eigenvalue for the piezoelectric tensor determines the highest piezoelectric coupling constant.
In this paper, we give two intervals to locate all $C$-eigenvalues for a given Piezoelectric-type tensor.
These intervals provide upper bounds for the largest $C$-eigenvalue.
Numerical examples are also given to show the corresponding results.

\end{abstract}

\end{frontmatter}


\section{Introduction}
Piezoelectric-type tensors are introduced by Chen et al. in \cite{Ch}
as a subclass of third order tensors which have extensive
applications in physics and engineering \cite{Cu,Ha,Kh,Lo,Ny,Zo}. The class of
Piezoelectric tensors, as the subclass of Piezoelectric-type tensors
of dimension three, plays the key role in Piezoelectric effect and
converse Piezoelectric effect \cite{Ch}.

\begin{definition} \label{Def-main} \cite[Definition 2.1]{Ch}
Let $\mathcal{A}=(a_{ijk})\in \mathbb{R}^{n\times n \times n}$ be a
third-order $n$ dimensional real tensor. If the later two indices of
$\mathcal{A}$ are symmetric, i.e., $a_{ijk}=a_{ikj}$ for all $j \in N$ and
$k\in N$ where $N:=\{1,2\ldots,n\}$, then $\mathcal{A}$ is called a piezoelectric-type tensor.
\end{definition}

To explore more properties related to piezoelectric effect and
converse piezoelectric effect in solid crystal, Chen et al. in
\cite{Ch} introduced  $C$-eigenvalues and $C$-eigenvectors for
Piezoelectric-type tensors, and shown that the largest
$C$-eigenvalue corresponds to the electric displacement vector with
the largest $2$-norm in the piezoelectric electronic effect under
unit uniaxial stress \cite{Ch,Cu,Zh}.

\begin{definition} \label{Def-main1} \cite[Definition 2.2]{Ch}
Let $\mathcal{A}=(a_{ijk})\in \mathbb{R}^{n\times n \times n}$ be a
piezoelectric-type tensor. If there exist a scalar $\lambda\in
\mathbb{R}$, vectors $x\in \mathbb{R}^n$ and $y\in \mathbb{R}^n$
satisfying the following system
\begin{equation} \label{Cha_equ} \mathcal{A}yy=\lambda x,~  x\mathcal{A}y=\lambda y,~ x^Tx=1~and
~y^Ty=1,\end{equation} where $\mathcal{A}yy \in \mathbb{R}^n$ and $x\mathcal{A}y \in \mathbb{R}^n$ with the $i$-th entry
\[ (\mathcal{A}yy)_i=\sum\limits_{j,k\in N}a_{ijk}y_jy_k, ~and ~ (x\mathcal{A}y)_i= \sum\limits_{j,k\in N}a_{jki}x_jy_k,\]
respectively, then $ \lambda$ is called a
$C$-eigenvalue of $\mathcal{A} $, $x$ and $y$ are called associated
left and right $C$-eigenvectors, respectively.
\end{definition}

For $C$-eigenvalues and associated left and right $C$-eigenvectors of
a piezoelectric-type tensor, Chen et al. in \cite{Ch} also provided
several related results, such as:

\begin{property} For a
piezoelectric-type tensor $\mathcal{A}$, there always exist
$C$-eigenvalues of $\mathcal{A}$ and associated left and right
$C$-eigenvectors. \end{property}

\begin{property} \label{pro2} Suppose that $\lambda$, $x$ and  $y$ are a
$C$-eigenvalue and its associated left and right $C$-eigenvectors of
a piezoelectric-type tensor $\mathcal{A}$. Then
\[\lambda =x\mathcal{A}yy,\]
where $x\mathcal{A}yy=\sum\limits_{i,j,k\in N}a_{ijk}x_iy_jy_k.$ Furthermore, $(\lambda, x,-y)$, $(-\lambda, -x,y)$ and
$(-\lambda, -x,-y)$ are also $C$-eigenvalues and their associated
$C$-eigenvectors of  $\mathcal{A}$.
\end{property}

\begin{property} \label{pro3} Suppose that $\lambda^*$ is the largest $C$-eigenvalue of a piezoelectric-type tensor
$\mathcal{A}$. Then
\[\lambda^* =\max\left\{x\mathcal{A}yy:~x^Tx=1, y^Ty=1\right\}.\]
\end{property}

Property \ref{pro2} and Property \ref{pro3} provide theoretically
the form to determine  $C$-eigenvalues or the largest $C$-eigenvalue
$\lambda^*$ of $\mathcal{A}$, However, it is difficult  to compute
them in practice because determining $x$ and $y$ is not easy. So, we
in this paper give some intervals to locate all  $C$-eigenvalues of
a piezoelectric-type tensor, and then give some upper bounds  for
the the largest $C$-eigenvalue. This can provide more information
before calculating them out.

\section{Main results}
In this section, we give two intervals to locate all $C$-eigenvalues
of a piezoelectric-type tensor. And the comparison of these two
intervals are also established.

\begin{thm} \label{Thm_main1}
Let $\mathcal{A}=(a_{ijk})\in \mathbb{R}^{n\times n \times n}$ be a
piezoelectric-type tensor, and $\lambda$ be a $C$-eigenvalue of
$\mathcal{A}$. Then
\begin{equation}\label{eq2.1} \lambda \in  \left[-\rho,~\rho \right], \end{equation}
where
\[ \rho :=\max\limits_{i,j \in N}\left( R_i^{(1)}(\mathcal{A})
R_j^{(3)}(\mathcal{A})\right)^{\frac{1}{2}},\]
$R_i^{(1)}(\mathcal{A}):=\sum\limits_{l,k\in N} |a_{ilk}|$ and
$R_j^{(3)}(\mathcal{A}):=\sum\limits_{l,k\in N} |a_{lkj}|$.
\end{thm}

\begin{proof}Suppose that $x=(x_1,x_2,\ldots,x_n)^T$ and
$y=(y_1,y_2,\ldots,y_n)^T$ are left and right $C$-eigenvectors
corresponding to $\lambda$ with $x^Tx=1$ and $y^Ty=1$. Let
\[ |x_p|=\max\limits_{i\in N} |x_i|,~ and~  |y_q|=\max\limits_{i\in N} |y_i|. \]
Then $0<|x_p|\leq 1 $ and  $0<|y_q|\leq 1$ because $x^Tx=1$ and
$y^Ty=1$.

By considering  the $p$-th equation of $\mathcal{A}yy=\lambda x$ in
(\ref{Cha_equ}), we have
\begin{equation}\label{eq2.2}  \lambda x_p = \sum\limits_{j,k\in N} a_{pjk} y_j y_k,\end{equation}
and
\begin{eqnarray*} |\lambda| |x_p| &\leq& \sum\limits_{j,k\in N} |a_{pjk}| |y_j|
|y_k|\\
& \leq & \sum\limits_{j,k\in N} |a_{pjk}| |y_q| |y_q|\\
& \leq & \sum\limits_{j,k\in N} |a_{pjk}| |y_q|.~ (by ~|y_q|\leq 1)
\end{eqnarray*}
Hence \begin{equation}\label{eq2.3} |\lambda| |x_p| \leq
R_p^{(1)}(\mathcal{A})|y_q|.\end{equation}

On the other hand, by considering  the $q$-th equation of
$x\mathcal{A}y=\lambda y$ in (\ref{Cha_equ}), we have
\begin{equation}\label{eq2.4}  \lambda y_q = \sum\limits_{i,j\in N} a_{ijq} x_i y_j,\end{equation}
and
\begin{eqnarray*} |\lambda| |y_q| &\leq& \sum\limits_{i,j\in N} |a_{ijq}| |x_i|
|y_j|\\
& \leq & \sum\limits_{i,j\in N} |a_{ijq}| |x_p|
|y_q|\\
& \leq & \sum\limits_{i,j\in N} |a_{ijq}| |x_p|.~(by ~|y_q|\leq 1) \\
\end{eqnarray*}
Hence \begin{equation}\label{eq2.5} |\lambda| |y_q|  \leq
R_q^{(3)}(\mathcal{A}) |x_p|.\end{equation} Multiplying
(\ref{eq2.3}) with (\ref{eq2.5})yields \[|\lambda|^2  |x_p||y_q|
\leq R_p^{(1)}(\mathcal{A}) R_q^{(3)}(\mathcal{A}) |x_p||y_q|,\]
consequently,\begin{equation}\label{eq2.6} |\lambda|\leq \left(
R_p^{(1)}(\mathcal{A}) R_q^{(3)}(\mathcal{A})\right)^{\frac{1}{2}}.
\end{equation}

Note the facts that  $\lambda$ is a $C$-eigenvalue of $\mathcal{A}$
if and only if  $-\lambda$ is a $C$-eigenvalue of $\mathcal{A}$, and
that a $C$-eigenvalue is real. Then
\[ \lambda\in \left[-\left(
R_p^{(1)}(\mathcal{A})
R_q^{(3)}(\mathcal{A})\right)^{\frac{1}{2}},~\left(
R_p^{(1)}(\mathcal{A}) R_q^{(3)}(\mathcal{A})\right)^{\frac{1}{2}}
\right]\subseteq [-\rho,~\rho].\] The conclusion follows.
\end{proof}

From Theorem \ref{Thm_main1}, we can obtain easily the following
upper bound for the largest $C$-eigenvalue of a piezoelectric-type
tensor.

\begin{corollary} \label{Cor1}
Let $\mathcal{A}=(a_{ijk})\in \mathbb{R}^{n\times n \times n}$ be a
piezoelectric-type tensor, and $\lambda^*$ be the largest
$C$-eigenvalue of $\mathcal{A}$. Then
\[ \lambda^* \leq \rho.\]
\end{corollary}

Next we give another interval to locate all $C$-eigenvalues of a
piezoelectric-type tensor. Before that some notation are given. For
a subset $S$ of $N$, denote
\[ \Delta_S:=\{(i,j): i \in S~ or ~j\in S \}\]
and
\[ \overline{\Delta}_S:=\{(i,j): i \notin S ~and ~j\notin S \}.\]
Given a piezoelectric-type tensor $\mathcal{A}=(a_{ijk})\in
\mathbb{R}^{n\times n \times n}$, let
\[R_j^{\Delta_S,(3)}(\mathcal{A})= \sum\limits_{(l,k)\in \Delta_S} |a_{lkj}|,
R_j^{\overline{\Delta}_S,(3)}(\mathcal{A})= \sum\limits_{(l,k)\in
\overline{\Delta}_S} |a_{lkj}|,\] where
$R_j^{\Delta_S,(3)}(\mathcal{A})=0$ if $S=\emptyset$, and
$R_j^{\overline{\Delta}_S,(3)}(\mathcal{A})=0$ if $S=N$. Obviously,
$R_j^{(3)}(\mathcal{A})=R_j^{\Delta_S,(3)}(\mathcal{A})+R_j^{\overline{\Delta}_S,(3)}(\mathcal{A})$
for each $j\in N$.

\begin{thm} \label{Thm_main2}
Let $\mathcal{A}=(a_{ijk})\in \mathbb{R}^{n\times n \times n}$ be a
piezoelectric-type tensor, and $\lambda$ be a $C$-eigenvalue of
$\mathcal{A}$. And let $S$ be a subset of $N$. Then
\begin{equation}\label{eqn2.1} \lambda \in  [-\rho_S,~\rho_S], \end{equation}
where
\[ \rho_S := \max\limits_{i,j\in N} \frac{1}{2} \left( R_j^{\Delta_S,(3)}(\mathcal{A})+\left( (R_j^{\Delta_S,(3)}(\mathcal{A}))^2
+4R_i^{(1)}(\mathcal{A}) R_j^{\overline{\Delta}_S,(3)}(\mathcal{A})
\right)^{\frac{1}{2}} \right ).\]

Furthermore,
\begin{equation}\label{eqnn2.1} \lambda \in  [-\rho_{min},~\rho_{min}], \end{equation}
where $\rho_{min} :=\min\limits_{S\subseteq N} \rho_S $.
\end{thm}

\begin{proof} Similarly to the proof of Theorem \ref{Thm_main1},
(\ref{eq2.3}) and (\ref{eq2.4}) hold. Furthermore, by (\ref{eq2.4})
we have
\begin{eqnarray*} |\lambda| |y_q| &\leq& \sum\limits_{i,j\in N} |a_{ijq}| |x_p|
|y_q|\\
& = & R_q^{(3)}(\mathcal{A})|x_p| |y_q|\\&=&
\left(R_q^{\Delta_S,(3)}(\mathcal{A})+R_q^{\overline{\Delta}_S,(3)}(\mathcal{A})\right)|x_p|
|y_q|\\&\leq& R_q^{\Delta_S,(3)}(\mathcal{A})
|y_q|+R_q^{\overline{\Delta}_S,(3)}(\mathcal{A})|x_p|
\end{eqnarray*}
Hence \begin{equation}\label{eq2.7}
\left(|\lambda|-R_q^{\Delta_S,(3)}(\mathcal{A})\right) |y_q|  \leq
R_q^{\overline{\Delta}_S,(3)}(\mathcal{A})|x_p|.\end{equation}
Multiplying (\ref{eq2.3}) with (\ref{eq2.7}) yields \[|\lambda|
\left(|\lambda|-R_q^{\Delta_S,(3)}(\mathcal{A})\right) |x_p||y_q|
\leq R_p^{(1)}(\mathcal{A})
R_q^{\overline{\Delta}_S,(3)}(\mathcal{A}) |x_p||y_q|,\]
consequently,\begin{equation}\label{eq2.8} |\lambda|
\left(|\lambda|-R_q^{\Delta_S,(3)}(\mathcal{A})\right) \leq
R_p^{(1)}(\mathcal{A}) R_q^{\overline{\Delta}_S,(3)}(\mathcal{A}).
\end{equation}
Solving (\ref{eq2.8}) for $ |\lambda|$ gives
\[|\lambda|\leq \frac{1}{2} \left( R_q^{\Delta_S,(3)}(\mathcal{A})+\left( (R_q^{\Delta_S,(3)}(\mathcal{A}))^2
+4R_p^{(1)}(\mathcal{A}) R_q^{\overline{\Delta}_S,(3)}(\mathcal{A})
\right)^{\frac{1}{2}} \right ).\] By an analogous way of Theorem
\ref{Def-main}, we have
\begin{equation}\label{eq2.9}  \lambda \in [-\rho_S,~\rho_S]. \end{equation}

Furthermore,  since (\ref{eq2.9}) holds for any $S\subseteq N$, it
follows that
\[ \lambda \in \bigcap\limits_{S\subseteq N } [-\rho_S,~\rho_S]=
\left[-\min\limits_{S\subseteq N}\rho_S, ~ \min\limits_{S\subseteq
N}\rho_S \right]= [-\rho_{min},~\rho_{min}].\] The conclusion
follows.
\end{proof}

Note here that if $S=\emptyset$, then
$R_j^{\Delta_S,(3)}(\mathcal{A})=0$ and
$R_j^{\overline{\Delta}_S,(3)}(\mathcal{A})= R_j^{(3)}(\mathcal{A})$
for any $j\in N$, which implies
\[\frac{1}{2} \left( R_j^{\Delta_S,(3)}(\mathcal{A})+\left( (R_j^{\Delta_S,(3)}(\mathcal{A}))^2
+4R_i^{(1)}(\mathcal{A}) R_j^{\overline{\Delta}_S,(3)}(\mathcal{A})
\right)^{\frac{1}{2}} \right ) =\left( R_i^{(1)}(\mathcal{A})
R_j^{(3)}(\mathcal{A}) \right)^{\frac{1}{2}}\]  consequently,
\[\rho_S= \rho.\]
Hence,
\[\rho_{min} = \min\limits_{S \subseteq N}  \rho_S \leq \rho.\] This
gives the comparison of the intervals in Theorem \ref{Thm_main1} and
Theorem \ref{Thm_main2} as follows.

\begin{thm} \label{Com}
Let $\mathcal{A}=(a_{ijk})\in \mathbb{R}^{n\times n \times n}$ be a
piezoelectric-type tensor, and $\lambda$ be a $C$-eigenvalue of
$\mathcal{A}$. Then
\[\lambda \in [-\rho_{min},~\rho_{min}]\subseteq [-\rho,~\rho],\]
where $ \rho$ is defined in Theorem \ref{Thm_main1}, and $\rho_{min}$ is defined in Theorem \ref{Thm_main2}.
\end{thm}

\begin{rmk}
Theorem \ref{Com} shows that the interval $[-\rho_{min},~\rho_{min}]
$ captures all $C$-eigenvalues of a piezoelectric-type tensor
precisely than the interval $[-\rho,~\rho]$, although $\rho_{min}$
needs more computations than $\rho $.
\end{rmk}

Similarly to Corollary \ref{Cor1}, we can obtain easily the
following upper bound for the largest $C$-eigenvalue of a
piezoelectric-type tensor by Theorem \ref{Thm_main2}.

\begin{corollary} \label{Cor2}
Let $\mathcal{A}=(a_{ijk})\in \mathbb{R}^{n\times n \times n}$ be a
piezoelectric-type tensor, and $\lambda^*$ be the largest
$C$-eigenvalue of $\mathcal{A}$. Then
\[ \lambda^* \leq \rho_{min}.\]
\end{corollary}

\section{Numerical examples}
In this section, we give some examples to show the results obtained above.
Consider the eight
piezoelectric tensors in \cite{Ch};

 (I) The piezoelectric tensor
$\mathcal{A}_{VFeSb}$ \cite{Ch,Jo}, with its entries
\[ a_{123}=a_{213}=a_{312}=-3.68180677,\]
and other elements are zeros;

 (II) The piezoelectric tensor
$\mathcal{A}_{SiO_2}$ \cite{Ch,Cu,Ha}, with its entries
\[ a_{111}=-a_{122}=-a_{212}=-0.13685,~and~a_{123}=-a_{213}=-0.009715,\]
and other elements are zeros;

 (III) The piezoelectric tensor
$\mathcal{A}_{Cr_2AgBiO_8}$ \cite{Ch,Jo}, with its entries
\[ a_{123}=a_{213}=-0.22163,~a_{113}=-a_{223}=2.608665,\]
\[ a_{311}=-a_{322}=0.152485,~and~a_{312}=-0.37153,\]
and other elements are zeros;

 (IV) The piezoelectric tensor
$\mathcal{A}_{RbTaO_3}$ \cite{Ch,Jo}, with its entries
\[ a_{113}=a_{223}=-8.40955,~a_{222}=-a_{212}=-a_{211}=-5.412525,\]
\[ a_{311}=a_{322}=-4.3031,~and~a_{333}=-5.14766,\]
and other elements are zeros;

(V) The piezoelectric tensor $\mathcal{A}_{NaBiS_2}$ \cite{Ch,Jo}, with its entries
\[ a_{113}=-8.90808,~a_{223}=-0.00842,~a_{311}=-7.11526,\]
\[ a_{322}=-0.6222,~and~a_{333}=-7.93831,\]
and other elements are zeros;

(VI) The piezoelectric tensor $\mathcal{A}_{LiBiB_2O_5}$ \cite{Ch,Jo},
with its entries
\[ a_{123}=2.35682,~a_{112}=0.34929,~a_{211}=0.16101,~a_{222}=0.12562,\]
\[ a_{233}=0.1361,~a_{213}=-0.05587,~a_{323}=6.91074,~and~a_{312}=2.57812,\]
and other elements are zeros;

(VII) The piezoelectric tensor $\mathcal{A}_{KBi_2F_7}$ \cite{Ch,Jo},
with its entries
\[ a_{111}=12.64393,~a_{122}=1.08802,~a_{133}=4.14350,~a_{123}=1.59052,\]
\[ a_{113}=1.96801,~a_{112}=0.22465,~a_{211}=2.59187,~a_{222}=0.08263,\]
\[ a_{233}=0.81041,~a_{223}=0.51165,~a_{213}=0.71432,~a_{212}=0.10570,\]
\[ a_{311}=1.51254,~a_{322}=0.68235,~a_{333}=-0.23019,~a_{323}=0.19013,\]
\[ a_{313}=0.39030,~and ~a_{312}=0.08381.\]

(VIII) The piezoelectric tensor $\mathcal{A}_{BaNiO_3}$ \cite{Ch,Jo},
with its entries
\[ a_{113}=a_{223}=0.038385,~a_{311}=a_{322}=6.89822,~and ~a_{333}=27.4628,\]
and other elements are zeros.

We  now use the intervals in Theorem \ref{Thm_main1} and Theorem
\ref{Thm_main2} to locate all $C$-eigenvalues of the eight tensors
above, see Table 1. It is easy to see that for any $C$-eigenvalue
$\lambda$,
\[ \lambda\in [-\rho_{min},~\rho_{min}] \subseteq [-\rho,~\rho].\]

 \hspace{-2cm} \noindent
\begin{tabular}[htbp]{c|c|c|c|c|c|c|c|c}\hline
                    & $\mathcal{A}_{VFeSb}$ &$\mathcal{A}_{SiO_2}$& $\mathcal{A}_{Cr_2AgBiO_8}$  &$\mathcal{A}_{RbTaO_3}$
                    & $\mathcal{A}_{NaBiS_2}$ & $\mathcal{A}_{LiBiB_2O_5}$   &$\mathcal{A}_{KBi_2F_7}$ &$\mathcal{A}_{BaNiO_3}$ \\\hline
        $\rho$       & 7.3636  & 0.2882  & 5.6606 &30.0911&17.3288&15.2911& 22.6896  & 38.8162   \\\hline
        $\rho_{min}$  & 7.3636  &  0.2834 &  5.6606 &23.5377&16.8548&12.3206&20.2351&   35.3787 \\\hline
        $\lambda^*$  & 4.2514  & 0.1375  & 2.6258 &12.4234&11.6674&7.7376&    13.5021   & 27.4628  \\\hline
\end{tabular}

\hspace{-1.cm} \noindent Table 1. The intervals $[-\rho,\rho]$ and
$[-\rho_{min},\rho_{min}]$, and $\lambda^*$ is the largest
$C$-eigenvalue.


\section*{Acknowledgements}
This work is partly supported by National Natural Science
Foundations of China (11601473) and CAS
"Light of West China" Program.







\end{document}